\documentclass[11pt,leqno]{article}

\usepackage{amsmath,amsthm,amssymb}
\usepackage{authblk,subfigure,booktabs,enumerate}
\usepackage{graphicx}
\usepackage{mathabx,bm}
\usepackage{url}
\usepackage{enumitem}
\usepackage[ruled,vlined]{algorithm2e}
\usepackage{geometry}
\theoremstyle{plain}
\newtheorem{thm}{Theorem}

\theoremstyle{remark}
\newtheorem*{rem}{Remark}

\theoremstyle{definition}
\newtheorem{ex}{Example}

\makeatletter

\makeatletter

\title{A deep learning approach to multi-marginal optimal transport \\ via Hilbert space embeddings of probability measures}

\author[1]{Yumiharu Nakano\thanks{Corresponding author. nakano@comp.isct.ac.jp}}
\author[1]{Takafumi Saito\thanks{saito.t.8e06@m.isct.ac.jp}}

\affil[1]{Department of Mathematical and Computing Science,
            Institute of Science Tokyo}

\date{\today}

\begin{document}

\maketitle

\begin{abstract}
We propose a numerical method for solving the multi-marginal Monge problem, which extends the classical Monge formulation to settings involving multiple target distributions.
Our approach is based on the Hilbert space embedding of probability measures and employs a penalization technique
using the maximum mean discrepancy to enforce marginal constraints.
The method is designed to be computationally efficient, enabling GPU-based implementation suitable for large-scale problems.
We confirm the effectiveness of the proposed method through numerical experiments using synthetic data.
\begin{flushleft}
\end{flushleft}
\begin{flushleft}
\end{flushleft}
\end{abstract}



\section{Introduction}\label{sec:1}

Our aim in this paper is to propose a numerical method for solving the multi-marginal optimal transport map problem,
formulated as follows. Let $\mu_1, \ldots, \mu_N$ be given Borel probability measures on $\mathbb{R}^d$. The multi-marginal Monge problem seeks a minimizer of
\begin{equation}
\label{eq:1}
\inf_{T \in \mathcal{T}(\mu_1, \ldots, \mu_N)} M(T),
\end{equation}
where
\[
M(T) = \int_{\mathbb{R}^d} c(T_1(x), \ldots, T_N(x))\, \mu_1(dx),
\]
for a given measurable cost function $c: \mathbb{R}^{Nd} \to [0, \infty)$, and $\mathcal{T}(\mu_1, \ldots, \mu_N)$ denotes the set of
all $N$-tuples $T = (T_1, \ldots, T_N)$ such that $T_1(x) = x$ for all $x \in \mathbb{R}^d$, and each $T_i: \mathbb{R}^d \to \mathbb{R}^d$
is Borel measurable and satisfies $T_i \sharp \mu_1 = \mu_i$ for $i = 2, \ldots, N$.
Here, $T\sharp \mu$ denotes the pushforward of $\mu$ with $T$, i.e., $T\sharp \mu(A)=\mu(T\in A)$ for any Borel set $A$.

Optimal transport theory, which provides a way to compute distances between probability measures, has recently gained increasing attention in areas such as machine learning.
Broadly, optimal transport theory consists of two main formulations: the Monge problem and the Monge-Kantorovich problem.
The Monge formulation seeks a deterministic map that transports one distribution to another while minimizing a given cost.
In contrast, the Monge-Kantorovich formulation relaxes this requirement by allowing transport plans, i.e., joint distributions with prescribed marginals,
thereby making the problem convex and more tractable.
A rich body of literature has since been developed around these formulations (see, e.g., \cite{san:2015, vil:2003, vil:2008} and the references therein).
Recently, \cite{sai-nak:2025} developed a GPU-based numerical solver for the Monge problem that is capable of handling large-scale data.
For the Monge-Kantorovich formulation, several practical Python libraries are already available, including \texttt{GeomLoss} \cite{Geom}, \texttt{KeOps} \cite{keO},
and the \texttt{OTT} library \cite{OTT}.

The focus of this paper is the multi-marginal Monge problem, which generalizes the classical Monge problem to multiple target marginals.
Although the \texttt{POT} library provides a solver for discrete multi-marginal problems, efficient numerical methods for continuous multi-marginal Monge maps remain limited.
This problem also underlies several related topics, such as Wasserstein barycenters \cite{puc-etal:2020}, which have important applications including shape morphing.
Libraries such as \texttt{POT} and \texttt{KeOps} also support numerical computation of Wasserstein barycenters.
This type of multi-marginal interpolation also finds applications in generative modeling, distributional robustness, and structured data alignment.

In this paper, we aim to develop a numerical method for the multi-marginal Monge problem in the setting where all state spaces have the same dimension.
Our primary goal is to construct an algorithm that supports large-scale GPU-based computations.
The proposed method is based on the Hilbert space embedding theory of probability measures,
which is applied by \cite{nak:2023} to the Schr{\"o}dinger bridge problem.
We adapt this framework to the multi-marginal Monge problem as in \cite{sai-nak:2025} and employ a penalization approach, using the maximum mean discrepancy as a penalty function
to enforce the marginal constraints.
The key advantages of our approach are: (i) computational efficiency, allowing for GPU-based implementation;
(ii) flexibility in handling non-Gaussian and structured distributions;
(iii) a theoretical guarantee for convergence under mild assumptions.
We demonstrate the effectiveness of the proposed method through numerical experiments on synthetic data involving both Gaussian and non-Gaussian marginals.

This paper is organized as follows. In Section \ref{sec:2}, we review the necessary background on the Hilbert space embedding of probability measures.
Section~\ref{sec:3} presents a general algorithm for the multi-marginal problem and several numerical experiments.
A theoretical convergence result is proved in Section~\ref{sec:4}.

\section{Hilbert space embeddings of probability measures}\label{sec:2}

We shall give a quick review of theory of Hilbert space embeddings of probability measures, as developed in \cite{sri-etal:2010}.
Denote by $\mathcal{P}(\mathbb{R}^d)$ the set of all Borel probability measures on $\mathbb{R}^d$.
Let $K$ be a symmetric and
 strictly positive definite kernel on $\mathbb{R}^{d}$, i.e.,
 $K(x,y)=K(y,x)$ for
 $x,y\in\mathbb{R}^{d}$ and
  for any positive distinct $x_{1},\ldots,x_{N}\in\mathbb{R}^{d}$ and
  $\alpha=(\alpha_{1},\ldots,\alpha_{N})  ^{\mathrm{T}}\in\mathbb{R}
  ^{N}\backslash\{0\}$,
\begin{equation*}
\sum^{N}_{j,\ell=1}
\alpha_{j}\alpha_{\ell}K(x_{j},x_{\ell})>0.
\end{equation*}
Assume further that $K$ is continuous on $\mathbb{R}^d\times\mathbb{R}^d$.
Then, there exists a unique Hilbert space
$\mathcal{H}$
 such that $K$ is a reproducing kernel on
 $\mathcal{H}$ with norm $\|\cdot\|$ (see, e.g., \cite{wen:2010}).
 Then consider the maximum mean discrepancy (MMD) $\gamma_{K}$ defined by
\[
 \gamma_{K}(\mu_{0},\mu_{1}):=
 \sup_{f\in\mathcal{H},\|f\|\leqq 1}
 \left|\int_{\mathbb{R}^{d}}f(x)\mu_{0}(dx)- \int_{\mathbb{R}^{d}}f(x)\mu_{1}(dx)\right|, \quad  \mu_{0},\mu_{1}\in\mathcal{P}(\mathbb{R}^{d}).
 \]
It is known that if $K$ is an integrally strictly positive definite, i.e.,
\[
 \int_{\mathbb{R}^d\times\mathbb{R}^d} K(x,y)\mu(dx)\mu(dy) >0
\]
for any finite and non-zero signed Borel measures $\mu$ on $\mathbb{R}^d$,
then $\gamma_K$ defines a metric on $\mathcal{P}(\mathbb{R}^d)$.
Examples of integrally strictly positive definite kernels include the Gaussian kernel
$K(x,y)=e^{-\alpha |x-y|^2}$, $x,y\in\mathbb{R}^d$, where $\alpha>0$ is a constant, and
the Mat{\'e}rn kernel $K(x,y)=K_{\alpha}(|x-y|)$, $x,y\in\mathbb{R}^d$, where $K_{\alpha}$ is the modified Bessel function of
order $\alpha>0$. It is also known that Gaussian kernel as well as Mat{\'e}rn kernel metrize the weak topology on $\mathcal{P}(\mathbb{R}^d)$.
See \cite{sri-etal:2010} and \cite{nak:2023}.

A main advantage of using $\gamma_K$ rather than other distances such as the Prokhorov distance,
the total variation distance, or the Wasserstein distance is that it is relatively easy to handle analytically due to its linear structure.
Indeed, by our boundedness assumption, for $\mu,\nu\in\mathcal{P}(\mathbb{R}^d)$
\[
 \gamma_K(\mu,\nu)=\left\|\int_{\mathbb{R}^d}K(\cdot,x)\mu(dx) - \int_{\mathbb{R}^d}K(\cdot,x)\nu(dx)\right\|_{\mathcal{H}},
\]
whence by the reproducing property,
\begin{equation}
\label{eq:gamma}
 \gamma_K(\mu,\nu)^2 = \int_{\mathbb{R}^d\times\mathbb{R}^d}K(x,y)(\mu-\nu)(dx)(\mu-\nu)(dy)
\end{equation}
(see Section 2 in \cite{sri-etal:2010}).
From this representation, given IID samples $X_1,\ldots,X_M\sim \mu$ and $Y_1,\ldots,Y_M\sim\nu$,
an unbiased estimator of $\gamma^2(\mu,\nu)$ is given by
\begin{equation}
\label{eq:2}
 \bar{\gamma}_K(\mu,\nu)^2 := \frac{1}{M(M-1)}\sum_i\sum_{j\neq i} K(X_i,X_j) - \frac{2}{M^2}\sum_{i,j}K(X_i,Y_j) + \frac{1}{M(M-1)}\sum_i\sum_{j\neq i}K(Y_i,Y_j)
\end{equation}
(see \cite{gre-etal:2006}).

\section{Algorithm and experiments}\label{sec:3}

\subsection{Deep learning algorithm}\label{sec:3.1}

We shall approximate $N$-tuple of transport maps by one of fully connected and feed-forward neural networks.
To this end, let $\sigma$ be a continuous and non-constant function on $\mathbb{R}$, which is the so-called activation function,
and  let $L\in\mathbb{N}$ be the number of hidden layers.
For a given $D=(D_1,\ldots,D_L)\in\mathbb{N}^L$, we denote by $\mathcal{N}_D$ the set of all mappings $G$ of the form
\[
 G(x) = A_{L+1}\circ \sigma \circ A_{L-1}\circ \cdots \circ \sigma \circ A_1(x), \quad x\in\mathbb{R}^d,
\]
where $A_{\ell}$'s are affine transformations defined by $A_{\ell}(\xi)=W_{\ell}\xi + b_{\ell}$, $\xi\in\mathbb{R}^{D_{\ell-1}}$,
for some $W_{\ell}\in\mathbb{R}^{D_{\ell}\times D_{\ell-1}}$ and $b_{\ell}\in\mathbb{R}^{D_{\ell}}$, $\ell=1,\ldots,L$.
Here we have set $D_0=D_{L+1}=d$ and denoted
$\sigma_i(\xi)=(\sigma_i(\xi_1),\ldots,\sigma_i(\xi_{D_{\ell}}))^{\mathsf{T}}$ for $\xi=(\xi_1,\ldots,\xi_{D_{\ell}})^{\mathsf{T}}\in\mathbb{R}^{D_{\ell}}$ by abuse of notation.
In practical applications, sample data is often given rather than an analytical form of the distribution,
so it is convenient to construct an approximate solution under a probabilistic formulation.  Consider the corresponding penalty method
\begin{equation}
\label{eq:3.1} \min M_{\lambda}(T) \;\;\text{s.t.}\;\; T_2,\ldots, T_N\in \mathcal{N}_D,
\end{equation}
with
\[
M_{\lambda}(T)=\mathbb{E}\left[c(X, T_2(X), \ldots,T_N(X))\right]
 + \sum_{i=2}^N\lambda_i\gamma_K(\mathbb{P}(T_i(X))^{-1}, \mu_i)^2
\]
where $\mathbb{E}$ is the expectation operator under a probability space $(\Omega,\mathcal{F},\mathbb{P})$,
$\lambda$ denotes $(N-1)$-tuple $(\lambda_2,\ldots,\lambda_N)$ with $\lambda_i>0$ for $i=2,\ldots,N$,
and $\mathbb{P}(T_i(X))^{-1}\in\mathcal{P}(\mathbb{R}^d)$ is the probability law of $T_i(X)$ under $\mathbb{P}$, i.e.,
$\mathbb{P}(T_i(X))^{-1}(A)=\mathbb{P}(T_i(X)\in A)$ for any Borel set $A\subset\mathbb{R}^d$.

To describe an algorithm based on a stochastic optimization for solving
\eqref{eq:3.1}, we represent $\mathcal{N}_D$ with parametrized set $\{T_{\theta}\}_{\theta\in\Theta}$ and
replace $\gamma_K$ with the sample-based one $\bar{\gamma}_K$ in \eqref{eq:2}.
Thus we are led to the following problem:
\begin{equation}
\label{eq:3}
 \min F(\theta) \;\;\text{s.t.}\;\; \theta\in\Theta
\end{equation}
with
\begin{align*}
F(\theta)&=\frac{1}{\prod^{N}_{i=2}\lambda_{i} M}\sum_{i=1}^M
c(X_{1,i},T_{2,\theta}(X_{1,i}), \ldots, T_{N,\theta}(X_{1,i})) \\
&\quad +\sum^{N}_{h=2}\frac{\lambda_{h}}{\prod^{N}_{i=2}\lambda_{i}}
\left(\frac{1}{M(M-1)}\sum_{i=1}^{M}\sum_{j\neq i}^{M}K(T_{h,\theta}(X_{1,i}),
T_{i,\theta}(X_{1,j})) - \frac{2}{M^{2}}\sum^{M}_{i,j=1}K(T_{h,\theta}(X_{1,i}),X_{h,j})\right),
\end{align*}
where $\{X_{h,i}\}_{i=1}^M$ is an i.i.d. with common distribution $\mu_h$ and
$T_{h,\theta}\in\mathcal{N}_D$ for each $h=1,\ldots,N$.

A learning algorithm for solving Problem~\eqref{eq:3} is summarized in Algorithm~\ref{alg:1}.
\begin{algorithm}[h]
\caption{Learning algorithm using empirical MMD}
\label{alg:1}
\DontPrintSemicolon

\KwData{Number of iterations $n$, batch size $M$, number of marginals $N$, weight parameters $\lambda_i > 0$ for $i = 2, \ldots, N$, learning rate $\eta$}
\KwResult{Estimated transport map $T$}

Initialize neural network parameters $\theta$\;
\For{$k = 1$ \KwTo $n$}{
    Draw i.i.d.\ samples $\{X_{h,i}\}_{i=1}^M$ from $\mu_h$ for each $h = 2, \ldots, N$\;

    Compute empirical objective $F(\theta)$ in Equation \eqref{eq:3} using the minibatch samples\;

    Update \(\theta\) by taking a gradient step:
    $\theta \leftarrow \theta - \eta \nabla_\theta F(\theta)$
}
\end{algorithm}

\subsection{Numerical experiments}\label{sec:3.2}

All numerical experiments are implemented in PyTorch and executed on a machine with an AMD EPYC 9654 CPU and 768~GB of RAM.
We conduct three experiments using synthetic datasets. In each experiment, the dataset size is fixed at $500$ samples.

The Gaussian kernel $K(x, y) = \exp(-|x - y|^2)$, cost function $c(x, y, z) = |x - y|^2 + |y - z|^2$, and the Adam optimizer with a learning rate of
$0.0001$ are used. The transport map $T(x)$ is modeled using a multi-layer perceptron (MLP) with two hidden layers.
All loss figures in this section use a logarithmic scale on the x-axis (epochs) for better visualization of convergence behavior.
Figures~\ref{fig:weight_gauss} and \ref{fig:weight_moon} also use a logarithmic scale on the y-axis.

\subsubsection{Interpolation of Shifted Gaussian Distributions}
\label{sec:3.1.2}

In this experiment, we consider three two-dimensional uncorrelated normal distributions. The initial distribution has mean $0$ and variance $1$,
the second has mean $3$ and standard deviation $1$, and the final has mean $10$ and standard deviation $1$.
The corresponding generated samples are shown in Figures~\ref{fig:gauss3} and \ref{fig:gauss10},
which correspond to the setting with Adam optimizer and $1/\lambda_2 = 1/\lambda_3 = 0.01$,
demonstrating that the proposed method accurately captures the target distributions.

\begin{figure}[h]
\centering
\begin{minipage}[b]{0.48\linewidth}
  \centering
  \includegraphics[width=60mm, bb = 0 0 461 346]{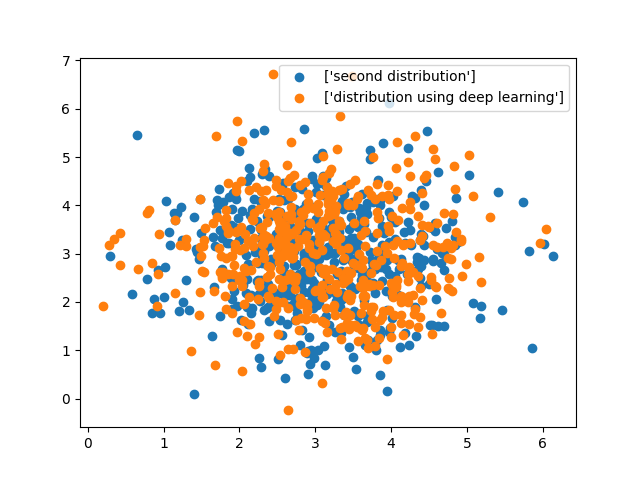}
  \caption{Generated vs. Target (Mean 3)}
  \label{fig:gauss3}
\end{minipage}
\begin{minipage}[b]{0.48\linewidth}
  \centering
  \includegraphics[width=60mm, bb = 0 0 461 346]{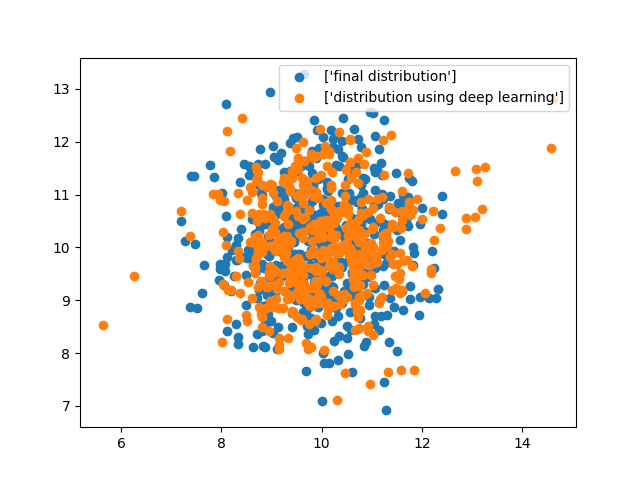}
  \caption{Generated vs. Target (Mean 10)}
  \label{fig:gauss10}
\end{minipage}
\end{figure}

Figure~\ref{fig:opt_gauss} shows the training loss across four optimizers. Adam and RMSprop achieve faster and more stable convergence compared to SGD and Adagrad.
Notably, convergence is typically reached within 2000 epochs for the former two.
Figure~\ref{fig:weight_gauss} displays the training loss for different values of the weight parameters $\lambda_2$ and $\lambda_3$.
Despite varying the weights over several orders of magnitude, the convergence behavior remains largely similar, indicating robustness with respect to this hyperparameter.

\begin{figure}[h]
\centering
\begin{minipage}[b]{0.49\linewidth}
  \centering
  \includegraphics[width=60mm, bb = 0 0 461 346]{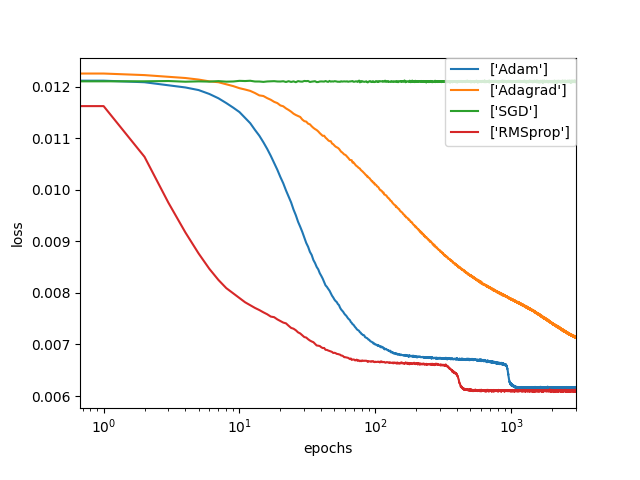}
  \caption{Loss vs. Optimizers.}
  \label{fig:opt_gauss}
\end{minipage}
\begin{minipage}[b]{0.49\linewidth}
  \centering
  \includegraphics[width=60mm, bb = 0 0 461 346]{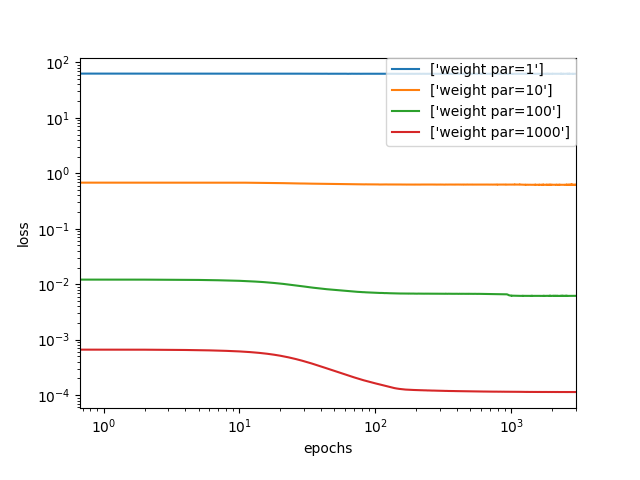}
  \caption{Loss vs. Weight Parameters}
  \label{fig:weight_gauss}
\end{minipage}
\end{figure}

Table~\ref{tab:optim} summarizes the mean and standard deviation of the generated samples under different optimizers.
Both Adam and RMSprop successfully produce samples close to the target mean values (3 and 10) with low variance, indicating stable and accurate transport.
In contrast, Adagrad and SGD perform poorly, particularly for the final distribution, where the generated means are significantly off-target and variances are large.
This supports the earlier observations from the loss curves in Figure~\ref{fig:opt_gauss}.
Table~\ref{tab:lambda} shows how the weight parameters $\lambda:=\lambda_2 = \lambda_3$ affect the sample statistics.
Among the tested values, $\lambda = 10$ and $\lambda = 100$ yield stable results, with both the means and variances close to the targets.
In contrast, the cases $\lambda = 1$ and $\lambda = 1000$ show significant deviations in both mean and variance.
These observations suggest that excessively small or large values of the weight parameters can degrade performance, and moderately large values are more
effective in enforcing the marginal constraints through the MMD penalty.

\begin{table}[h]
\centering
\caption{Performance comparison across optimizers (mean, standard deviation).}
\label{tab:optim}
\begin{tabular}{lcc}
\hline
Optimizer & Second (mean, SD) & Final (mean, SD) \\
\hline
Adam      & 3.0374, 1.0036     & 10.0204, 1.0103 \\
Adagrad   & 1.9073, 1.8980     & 0.0321, 1.8998 \\
SGD       & 0.1050, 0.0262     & 0.0032, 0.0420 \\
RMSprop   & 2.9813, 1.0012     & 9.9636, 1.0186 \\
\hline
\end{tabular}
\end{table}

\begin{table}[h]
\centering
\caption{Performance comparison across weight parameters (mean, standard deviation).}
\label{tab:lambda}
\begin{tabular}{lcc}
\hline
\(\lambda_2 = \lambda_3\) & Second (mean, SD) & Final (mean, SD) \\
\hline
1     & 3.0379, 1.0491    & -10.3407, 23.0344 \\
10    & 2.9772, 1.0889    & 9.9741, 1.0575   \\
100   & 3.0374, 1.0036     & 10.0204, 1.0103   \\
1000  & 3.0833, 0.9737    & -9.5127, 17.1537   \\
\hline
\end{tabular}
\end{table}

\subsubsection{Interpolation of Gaussian and Two Non-Gaussian Structured Distributions}
\label{sec:3.1.1}

In this experiment, the initial distribution is a two-dimensional standard normal distribution. The second distribution is the well-known two-moons dataset,
and the final distribution is the two-circles dataset \cite{scikit-learn}.
Results are typically obtained after approximately 5000 training epochs.

\begin{figure}[h]
\centering
\begin{minipage}[b]{0.49\linewidth}
  \centering
  \includegraphics[width=60mm, bb = 0 0 461 346]{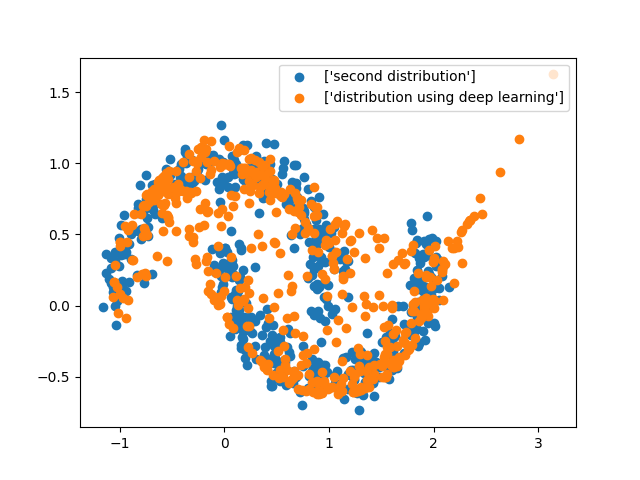}
  \caption{Generated vs. Target (Two-Moons)}
  \label{fig:moon}
\end{minipage}
\begin{minipage}[b]{0.49\linewidth}
  \centering
  \includegraphics[width=60mm, bb = 0 0 461 346]{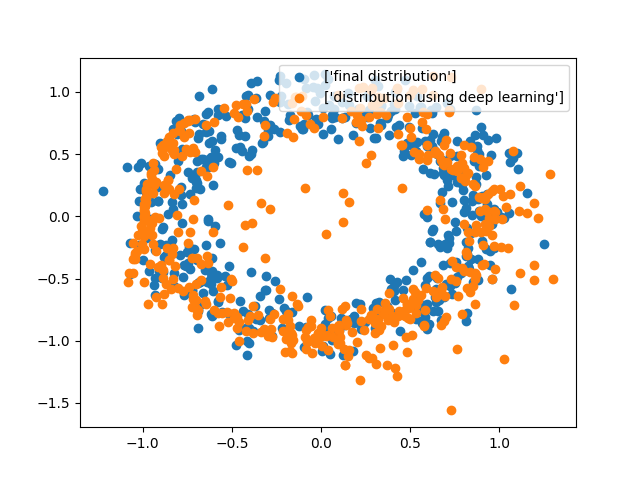}
  \caption{Generated vs. Target (Two-Circles)}
  \label{fig:circle}
\end{minipage}
\end{figure}

Figures~\ref{fig:moon} and~\ref{fig:circle} show the output of the proposed method using the Adam optimizer and regularization parameters
$1/\lambda_2 = 1/\lambda_3 = 0.01$. The generated samples closely match the target distributions.

\begin{figure}[h]
\centering
\begin{minipage}[b]{0.49\linewidth}
  \centering
  \includegraphics[width=60mm, bb = 0 0 461 346]{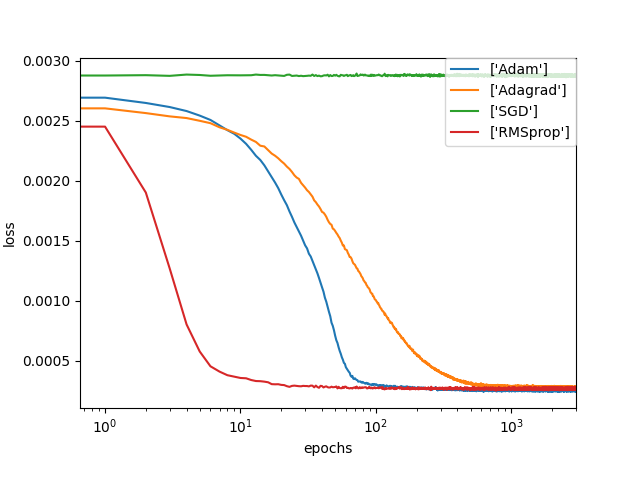}
  \caption{Loss vs.~Optimizers.}
  \label{fig:opt_moon}
\end{minipage}
\begin{minipage}[b]{0.49\linewidth}
  \centering
  \includegraphics[width=60mm, bb = 0 0 461 346]{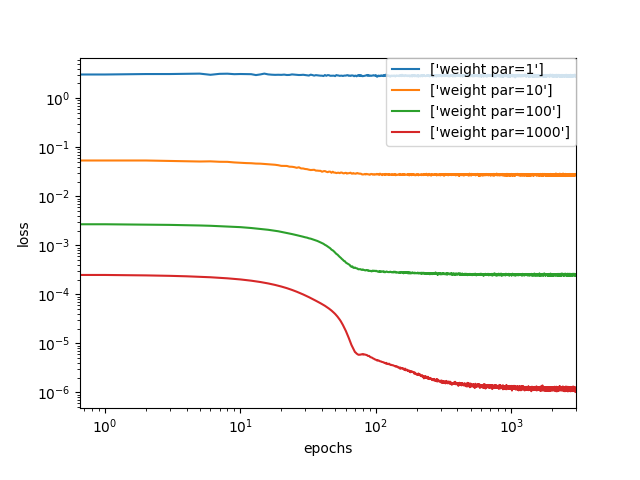}
  \caption{Loss vs.~Weight parameters.}
  \label{fig:weight_moon}
\end{minipage}
\end{figure}

Figure~\ref{fig:opt_moon} reports the loss curves for the shifted Gaussian interpolation task.
Again, Adam and RMSprop exhibit better performance in terms of both convergence speed and final loss values, while SGD and Adagrad show slower convergence.
Figure~\ref{fig:weight_moon} presents the training loss under different weight parameters in the Gaussian case.
Although slight variations in convergence speed are observed, the overall loss curves remain comparable.

\section{Convergence result}\label{sec:4}

In this section we show a theoretical convergence result.
To this end, we impose the following conditions on $\mu_i$'s, the kernel $K$, and the activation function $\sigma$:
\begin{enumerate}
\item[(A1)] For each $i=1,\ldots,N$, the probability measure $\mu_i$ vanishes on $(d-1)$-rectifiable sets and satisfies
\[
 \int_{\mathbb{R}^d}|x|^2 \mu_i(dx) < \infty.
\]
\item[(A2)] $K\in C(\mathbb{R}^d\times\mathbb{R}^d)$ is bounded and Lipschitz continuous on $\mathbb{R}^d\times\mathbb{R}^d$.
\item[(A3)] $\gamma_K$ is a metric on $\mathcal{P}(\mathbb{R}^d)$ and metrizes the weak topology on $\mathcal{P}(\mathbb{R}^d)$.
\item[(A4)] The function $\sigma$ is continuous and non-constant on $\mathbb{R}$ such that
\[
 |\sigma(r)|\le C(1+|r|^2), \quad r\in\mathbb{R},
\]
for some positive constant $C$.
\end{enumerate}

\begin{rem}
It is well-known that if $\mu_i$ is absolutely continuous with respect to the Lebesgue measure then it vanishes on
$(d-1)$-rectifiable sets.
\end{rem}

\begin{rem}
As stated in Section \ref{sec:2}, Gaussian and Mat{\'e}rn kernels satisfy (A2) and (A3).
\end{rem}

\begin{rem}
If $\sigma$ is the so-called ReLU activation function defined by $\sigma(x)=\max\{x,0\}$,
then the condition (A4) is satisfied.
\end{rem}

Then we shall prove the convergence of our deep learning algorithm in the case where the cost function $c$ is given by
\begin{equation}
\label{eq:5}
 c(x_1,\ldots,x_N) = \sum_{i<j}|x_i -x_j|^2, \quad x_1,\ldots,x_N\in\mathbb{R}^d.
\end{equation}
In this case, the problem \eqref{eq:1} is equivalent to
\[
 \sup\left\{\mathbb{E}\left[\sum_{i<j}X_i^{\mathsf{T}}X_j\right]: X_i\sim\mu_i, \; 1\le i\le N\right\}
\]
as well as
\[
 \sup\left\{\mathbb{E}\left|\sum_{i=1}^{N}X_i\right|^2 :  X_i\sim\mu_i, \; 1\le i\le N\right\}.
\]
Moreover, this problem is equivalent to the one of finding Wasserstein barycenters and
\[
 \inf\left\{\mathbb{E}\sum_{i=1}^N|X_i-\overline{S}|^2 :   X_i\sim \mu_i, \; 1\le i\le N\right\}.
\]
In particular, if an $N$-tuple $(X_1,\ldots,X_N)$ is optimal then the distribution of $\overline{S}=\frac{1}{N}\sum_{i=1}^NX_i$ is
a Wasserstein barycenter. We refer to \cite{puc-etal:2020} for details.

Under (A1) there exists a unique solution
$T^*=(T_1^*,\ldots, T_N^*)\in\mathcal{T}(\mu_1,\ldots,\mu_N)$ to the Monge problem \eqref{eq:1} with cost function $c$ given by \eqref{eq:5}.
See \cite{gan-swi:1998} and \cite{rus-uck:2002}.

In the algorithm described in Section \ref{sec:3}, the numbers $D_{\ell}$'s in hidden layers are fixed.
Here, to prove convergence we assume that $D=(D_1,\ldots,D_L)$ can be arbitrary chosen.
So denote $N=\bigcup_{D\in\mathbb{N}^L}\mathcal{N}_D$ and let $\mathcal{N}^*$ be the set of all
$N$-tuple $T=(T_1,\ldots,T_N)$ such that $T_1(x)=x$ for all $x\in\mathbb{R}^d$ and
$T_i\in\mathcal{N}$ for each $i=2,\ldots,N$.

For $i=2,\ldots, N$, define
\begin{align*}
K_i(x,y):=K(x,y)-\int_{\mathbb{R}^{d}}K(x,y^{\prime})
\mu_{1}(dy^{\prime}) -\int_{\mathbb{R}^{d}} K(x^{\prime},y)\mu_{1}(dx^{\prime}).
\end{align*}
Then, by \eqref{eq:gamma}, if two random variables $X$ and $\tilde{X}$ on some probability space $(\Omega,\mathcal{F},\mathbb{P})$
are mutually independent and both follow $\mu$, then we have
\[
 \gamma_K(\mu,\mu_i)^2= \mathbb{E}_{\mathbb{P}}\left[K_i(X,\tilde{X})\right]
  + \int_{\mathbb{R}^d\times\mathbb{R}^d} K(x^{\prime},y^{\prime})\mu_i(dx^{\prime})\mu_1(dy^{\prime}).
\]
Consequently, our problem is described as
\begin{equation}
\label{eq:7}
 \inf_{T\in\mathcal{N}^*}M_{\lambda}(T).
\end{equation}
Here, note that in the present setting, the cost $M_{\lambda}(T)$ is given by
\[
 M_{\lambda}(T)= \mathbb{E}\left[\sum_{i<j}|T_i(X)-T_j(X)|^2  + \sum_{i=2}^N\lambda_i K_i(T_i(X), T_i(\tilde{X}))\right]
  + \sum_{i=2}^N\lambda_i\int_{\mathbb{R}^d\times\mathbb{R}^d} K(x^{\prime},y^{\prime})\mu_i(dx^{\prime})\mu_i(dy^{\prime}),
\]
where $T_1(x)=x$ for $x\in\mathbb{R}^d$.
Note also that under (A1)--(A4) we have $0\le M(T)<\infty$ for any $T\in\mathcal{N}^*$.

Let $\varepsilon^{h}$ and $\lambda_i^{h}$, $i=1,\ldots,N$, be positive functions of $h\in (0,1)$ such that
\[
 \varepsilon^h\searrow 0, \quad \lambda_i^h\nearrow +\infty, \;\; i=1,\ldots,N, \quad h\searrow 0.
\]
Then for each $h$ take $\varepsilon^h$-optimal map $T^h$ to
the problem \eqref{eq:7} with $\lambda=\lambda^h$.
Thus $T^h=(T_1^h,\ldots,T_N^h)\in\mathcal{N}^*$ satisfies
\[
 M_{\lambda^h}(T^h) - \varepsilon^h \le \inf_{T\in\mathcal{N}^*}M_{\lambda^h}(T).
\]

Let $L^2(\mu_1)$ the set of all Borel measurable mappings $f:\mathbb{R}^d\to\mathbb{R}^d$ such that
\[
 \|f\|:= \sqrt{\int_{\mathbb{R}^d}|f(x)|^2\mu_1(dx)}<\infty.
\]
Then we impose the following:
\begin{enumerate}
\item[(A5)] $\mathcal{N}$ is dense in $L^2(\mu_1)$ with respect to the norm $\|\cdot\|$.
\end{enumerate}

\begin{ex}
\begin{enumerate}
\item In the case of $L=1$, if $\sigma$ is bounded, then the condition (A5) hold by the classical universal approximation theorem given in \cite[Theorem 1]{hor:1991}.
\item In general cases of $L\ge 1$, if $\sigma$ is the ReLU activation function
then the condition (A5) hold by Theorem 4.16 in \cite{kid-lyo:2020}. Actually this is an universal approximation theorem for
$L^2(\mathbb{R}^d;\mathbb{R}^d)$ rather than $L^2(\mu_1)$, but the arguments in the proof can be applied to the space $L^2(\mu_1)$.
\end{enumerate}
\end{ex}

Here is our main convergence result.
\begin{thm}
\label{thm:1}
Let $c$ be given by $\eqref{eq:5}$.
Suppose that $(A1)$--$(A5)$ hold.
Then we have
\begin{equation}
\label{eq:8}
 \lim_{h\searrow 0}\sum_{i=2}^N\lambda_i^h\gamma_K(T_i^h\sharp \mu_1 ,\mu_i)^2 =0.
\end{equation}
In particular, $T_i^h\sharp \mu_1$ weakly converges to $\mu_i$ as $h\searrow 0$ for each $i=2,\ldots,N$.
\end{thm}

\begin{rem}
Theorem~\ref{thm:1} guarantees the asymptotic satisfaction of the marginal constraints enforced via the MMD penalty.
This result ensures that the learned maps asymptotically push forward the source distribution to the prescribed targets in distribution.
We note that the convergence of the transport cost value is not addressed here and remains an open direction for further analysis.
\end{rem}

\begin{proof}[Proof of Theorem $\ref{thm:1}$]
Step (i). First we will show that
\begin{equation}
\label{eq:9}
  \inf_{T\in\mathcal{N}^*}M_{\lambda}(T)\le M(T^*)
\end{equation}
for any $\lambda=(\lambda_2,\ldots,\lambda_N)$ with $\lambda_i>0$, $i=2,\ldots,N$.
To this end, let $T_i^*$ be such that $T^*=(T_1^*,\ldots,T_N^*)$. Since each $T_i^*\in L^2(\mu_1)$, by (A5), for any fixed $\varepsilon>0$ there exist
$\tilde{T}_2,\ldots,\tilde{T}_N\in\mathcal{N}$ such that
\[
 \sqrt{\mathbb{E}|\tilde{T}_i(X) - T^*_i(X)|^2}\le \varepsilon, \quad i=2,\ldots,N.
\]
It is straightforward to see that
\begin{align*}
 &|\mathbb{E}[c(X,\tilde{T}_2(X),\ldots,\tilde{T}_N(X))] - \mathbb{E}[c(X,T_2^*(X),\ldots,T_N^*(X))]| \\
 &\le 3\sum_{i,j=1}^N\mathbb{E}[(|\tilde{T}_i(X)| + |T_i^*(X)|)^2]^{1/2}\mathbb{E}[|\tilde{T}_j(X) - T_j^*(X)|^2]^{1/2},
\end{align*}
where $\tilde{T}_1(x)=T_1^*(x)=x$.
Further, by (A2), for $x,\tilde{x},y,\tilde{y}\in\mathbb{R}^d$,
\begin{align*}
&|K_i(x,y) - K_i(\tilde{x},\tilde{y})| \\
&\le |K(x,y) - K(\tilde{x},\tilde{y})| + \int_{\mathbb{R}^d}|K_i(x,y^{\prime}) - K_i(\tilde{x},y^{\prime})| \mu_1(dy^{\prime})
     + \int_{\mathbb{R}^d}|K_i(x^{\prime},y) - K_i(x^{\prime}, \tilde{y})| \mu_1(dx^{\prime}) \\
&\le 2C_K(|x-\tilde{x}| + |y-\tilde{y}|),
\end{align*}
where $C_K$ is the Lipschitz coefficient of $K$. Thus each $K_i$ is again Lipschitz continuous.
This together with $\gamma_K(T_i^*\sharp\mu_1, \mu_i)=0$ yield
\begin{align*}
  \gamma_K(\tilde{T}_i\sharp\mu_1, \mu_i)^2 &= \gamma_K(\tilde{T}_i\sharp\mu_1, \mu_i)^2 - \gamma_K(T_i^*\sharp\mu_1, \mu_i)^2
 = \mathbb{E}[K_i(\tilde{T}_i(X), \tilde{T}_i(\tilde{X})) - K_i(T^*_i(X), T^*_i(\tilde{X}))] \\
 &\le 4C_K\mathbb{E}|\tilde{T}_i(X)-T_i^*(X)|.
\end{align*}
Therefore, there exists a positive constant $C$ such that
\begin{align*}
 &M_{\lambda}(\tilde{T})- M(T^*) \\
 &= \mathbb{E}[c(X,\tilde{T}_2(X),\ldots,\tilde{T}_N(X))] - \mathbb{E}[c(X,T_2^*(X),\ldots,T_N^*(X))]
   + \sum_{i=2}^N\lambda_i\gamma_K(\tilde{T}_i\sharp\mu_1, \mu_i)^2 \\
 &\le C\varepsilon.
\end{align*}
Thus \eqref{eq:9} follows.

Step (ii).
We will show \eqref{eq:8}.
This claim can be proved by almost the same argument as that given in the proof of Theorem 3.1 in \cite{nak:2023}, but we shall give a proof for reader's convenience.
Assume contrary that
\begin{equation*}
 \limsup_{n\to\infty}\sum_{i=2}^N\lambda_i^{n}\gamma(T_i^{n}\sharp \mu_1,\mu_i)^2=5\delta
\end{equation*}
for some positive sequence $\{h_n\}_{n=1}^{\infty}$ such that $\lim_{n\to\infty}h_n=0$ and
for some $\delta>0$. Here we have denoted $\lambda_i^{n}=\lambda_i^{h_n}$ and $T_i^n=T_i^{h_n}$ for notational simplicity.
Then there exists a subsequence $\{n_k\}$ such that
\begin{equation*}
 \lim_{k\to\infty}\sum_{i=2}^N\lambda_i^{n_k}\gamma(T_i^{n_k}\mu_1,\mu_i)^2=5\delta.
\end{equation*}
Since $\gamma$ is a metric on $\mathcal{P}(\mathbb{R}^d)$, we have $\gamma(T^*_i\sharp\mu_1,\mu_i)=0$,
whence by \eqref{eq:9} $M^*_{\lambda}:=\inf_{T\in\mathcal{N}^*}M_{\lambda}(T)\le M(T^*)$ for any $\lambda$. This means
\begin{equation*}
 M_{\lambda^n}(T^n)\le M^*_{\lambda^n}+\varepsilon^n \le M(T^*)+\varepsilon^n,
\end{equation*}
where $\varepsilon^n=\varepsilon^{h_n}$.
Thus, the sequence $\{M_{\lambda^n}(T^n)\}_{n=1}^{\infty}$ is bounded, and so
there exists a further subsequence $\{n_{k_m}\}$ such that
\begin{equation*}
 \lim_{m\to\infty}M_{\bar{\lambda}^m}(\bar{T}^m)=\kappa:=\limsup_{k\to\infty} M_{\lambda^{n_k}}(T^{n_k}) < \infty,
\end{equation*}
where $\bar{\lambda}^m=\lambda^{n_{k_m}}$ and $\bar{T}^m=T^{n_{k_m}}$.
Now choose $m_0$ and $m_1$ such that
$\kappa< M_{\bar{\lambda}^{m_0}}(\bar{T}^{m_0}) + \delta$, $M_{\bar{\lambda}^{m_1}}(\bar{T}^{m_1})<\kappa +\delta$,
$\bar{\lambda}^{m_1}_i> 7\bar{\lambda}_i^{m_0}$, $i=2,\ldots,N$, and that
$3\delta + \bar{\varepsilon}_{m_0}< \sum_{i=2}^N\bar{\lambda}_i^{m_1}(\bar{\gamma}_i^{m_1})^2< 7\delta$,
where $\bar{\varepsilon}_m=\varepsilon_{n_{k_m}}$ and $\bar{\gamma}_i^m=\gamma(\bar{T}^m\sharp\mu_1,\mu_i)$.
With these choices it follows that
\begin{align*}
 \kappa &< M_{\bar{\lambda}^{m_0}}(\bar{T}^{m_0}) + \delta
 \le M^*_{\bar{\lambda}^{m_0}} + \bar{\varepsilon}_{m_0} + \delta
 \le M_{\bar{\lambda}^{m_0}}(\bar{T}^{m_1}) + \bar{\varepsilon}_{m_0} + \delta \\
 &= M(\bar{T}^{m_1}) + \sum_{i=2}^N\frac{\bar{\lambda}_i^{m_0}}{\bar{\lambda}_i^{m_1}}\bar{\lambda}_i^{m_1}(\bar{\gamma}_i^{m_1})^2 + \bar{\varepsilon}_{m_0} + \delta \\
 &< M(\bar{T}^{m_1}) + \frac{1}{7}\sum_{i=2}^N\bar{\lambda}_i^{m_1}(\bar{\gamma}_i^{m_1})^2 + \bar{\varepsilon}_{m_0}  + \delta
 < M(\bar{T}^{m_1}) + 2\delta + \bar{\varepsilon}_{m_0} \\
 &< M(\bar{T}^{m_1}) + \sum_{i=2}^N\bar{\lambda}_i^{m_1}(\bar{\gamma}_i^{m_1})^2 - \delta
 = M_{\bar{\lambda}^{m_1}}(\bar{T}^{m_1}) - \delta < \kappa,
\end{align*}
which is impossible.
\end{proof}

\bibliographystyle{hplain}
\bibliography{multi_ot}


\end{document}